\PassOptionsToPackage{monochrome}{xcolor}
\documentclass[sn-mathphys,Numbered]{sn-jnl}

\usepackage{graphicx}%
\usepackage{multirow}%
\usepackage{amsmath,amssymb,amsfonts}%
\usepackage{amsthm}%
\usepackage{mathrsfs}%
\usepackage[title]{appendix}%
\usepackage{xcolor}%
\usepackage{textcomp}%
\usepackage{manyfoot}%
\usepackage{booktabs}%
\usepackage{bookmark}
\usepackage{algorithm}%
\usepackage{algorithmicx}%
\usepackage{algpseudocode}%
\usepackage{listings}%
\usepackage{color,latexsym}
\usepackage{amsbsy}
\usepackage{amsmath}
\usepackage{bbm}
\usepackage{colortbl,dcolumn}
\usepackage{enumerate}
\usepackage{indentfirst}
\usepackage[utf8]{inputenc}
\usepackage{float}
\usepackage{subfigure}
\usepackage{psfrag}
\usepackage{verbatim}
\usepackage{placeins}

\def\R{\mathbb{R}}
\def\E{\mathbb{E}}
\def\N{\mathbb{N}}
\def \e{\text{e}}
\def \half{\tfrac12}
\def \d{{\rm d}}

\newtheorem{assumption}{Assumption}[section]
\newtheorem{remark}[assumption]{Remark}
\newtheorem{lemma}[assumption]{Lemma}
\newtheorem{theorem}[assumption]{Theorem}   
\newtheorem{proposition}[assumption]{Proposition}
\newtheorem{definition}[assumption]{Definition}
\newtheorem{example}[assumption]{Example}

\numberwithin{equation}{section}

\begin{document}

\title{An unconditional boundary and dynamics preserving scheme for the stochastic epidemic model 
}

\author[1]{\fnm{Ruishu} \sur{Liu}}\email{chicago@mail.ustc.edu.cn}
\equalcont{These authors contributed equally to this work.}

\author[1]{\fnm{Xiaojie} \sur{Wang}}\email{x.j.wang7@csu.edu.cn}
\equalcont{Three authors contributed equally to this work.}

\author*[1]{\fnm{Lei} \sur{Dai}}\email{dailei@csu.edu.cn}
\equalcont{These authors contributed equally to this work.}

\affil[1]{\orgdiv{School of Mathematics and Statistics}, \orgname{Central South University}, \orgaddress{ \city{Changsha}, \postcode{410000}, \state{HNP-LAMA}, \country{China}}}


\abstract{
    In the present article, we construct a logarithm transformation based Milstein-type method for the stochastic susceptible-infected-susceptible (SIS) epidemic model evolving in the  domain $(0,N)$.
    The new scheme is explicit and unconditionally boundary and dynamics preserving,
    when used to solve the stochastic SIS epidemic model.
    Also, it is proved that the scheme has a strong convergence rate of order one. 
    {\color{black}
    Different from existing time discretization schemes, the newly proposed scheme for any time step size $h>0$, not only produces numerical approximations living in the entire domain $(0, N)$, but also unconditionally reproduces the extinction and persistence behavior of the original model, with no additional requirements imposed on the model parameters.
    }
    Numerical experiments are presented to verify our theoretical findings.
}




\maketitle


\section{Introduction}\label{2023SIS-section:Introduction}
Over the past few decades, mathematical models have been constructed to describe the evolution of different kinds of epidemics. 
The classical Kermack-McKendrick model \cite{1927A} was used for modelling common childhood diseases, where a typical individual starts off susceptible, at some stage catches the disease, and after a short infectious period becomes permanently immune. 
Nevertheless, {\color{black}in terms of many unstable sources of infection, particularly most viruses,} cured people can rarely get long-time immunity and {\color{green}turn} to be susceptible to some mutations of the same original virus.
For these cases, the susceptible-infected-susceptible (SIS) model was constructed in \cite{2010Hethcote}. 
{\color{black}By adding randomness to the model, 
the classical susceptible-infected-susceptible epidemic model was extended by Gray et al. in \cite{gray2011stochastic} from a deterministic framework to a stochastic one}, which is formulated as the following stochastic differential equations (SDEs) of It\^{o} type:
\begin{align}
    \d S_t 
    & = 
    \left[ \mu N - \beta S_t I_t + \gamma I_t - \mu S_t \right] \, \d t 
    - 
    \sigma S_t I_t  \, \d W_t \, ,
    \label{2023SIS-eq:SIS_original-system1}\\
    \d I_t
    & =
    \left[ \beta S_t I_t - (\mu + \gamma)I_t \right]\, \d t
    +
    \sigma S_t I_t \, \d W_t \, ,
    \label{2023SIS-eq:SIS_original-system2}
\end{align}
where $S_t$ denotes the number of susceptibles and $I_t$ the number of infecteds at time $t \in [0,T]$, with initial values satisfying 
$S_0 + I_0 = N$.
{\color{black}
In this model, the total population $N$ is divided into susceptibles $S_t$ and infecteds $I_t$, with the assumption that recovered people become susceptible again. 
Here and below, the model parameters $\mu$, $\beta$, $\gamma$, $\sigma > 0$ are 
assumed to be positive constants.
More accurately, $\mu$ stands for the per capita death rate, $\beta$ is the disease transmission coefficient, $\gamma$ means the cure rate 
and $\sigma$ represents the variance of the occurrence of potentially infectious contacts between infecteds and susceptibles. 
}
In addition, $\{ W_t \}_{t \in [0, T]}$, $T>0$ is a standard Brownian motion defined on a complete probability space $(\Omega, \mathcal{F}, \mathbb{P})$.
{\color{black}
Since ${\rm d}\, [ S_t + I_t ] = [\mu N - \mu (S_t + I_t)] \,{\rm d} t$, it can be directly deduced that $S_t + I_t = N$.
Thus the original SDE system \eqref{2023SIS-eq:SIS_original-system1}-\eqref{2023SIS-eq:SIS_original-system2} can be reduced to a scalar SDE for $I_t$:
\begin{equation}\label{2023SIS-eq:SIS_original}
    {\rm d} I_t 
    = 
    I_t ( \beta N - \mu - \gamma - \beta I_t ) \, {\rm d} t 
    + 
    \sigma I_t ( N - I_t ) \, {\rm d} W_t , \quad t \in [0,T],
    \quad
    I_0 \in (0,N).
\end{equation}
}
It was asserted in \cite{gray2011stochastic} that SDE \eqref{2023SIS-eq:SIS_original} admits a unique global solution in $(0,N)$, whose dynamics behavior was also discussed. On the one hand, 
{\color{black}
the extinction property
}
of SDE \eqref{2023SIS-eq:SIS_original},
{\color{black}
which means that the 
disease will die out with probability one (see Definition \ref{2023SIS-def:extinction}),
}
was confirmed under the conditions that
\begin{itemize}
    \item $R_0^S < 1$,
    \item $\sigma^2 \leq \tfrac{\beta}{N}$ or $\sigma^2 > \tfrac{\beta}{N} \vee \tfrac{\beta^2}{2(\mu + \gamma)}$,
\end{itemize}
where 
    \begin{align}
    R_0^D : = \tfrac{\beta N}{\mu + \gamma}
    \quad
    \text{and}
    \quad
    R_0^S : =
    R_0^D - \tfrac{\sigma^2 N^2}{2(\mu + \gamma)}
    \end{align}
are the basic reproduction numbers for the deterministic and stochastic SIS epidemic models respectively.
{\color{black}
On the other hand, SDE \eqref{2023SIS-eq:SIS_original} was shown to possess the persistence property, that is, the 
disease will stay (see Definition \ref{2023SIS-def:persistence}),  
when $ R_0^S > 1$.
}
Since the analytic solution of SDE \eqref{2023SIS-eq:SIS_original} is not available, people resort to numerical solutions that are expected to be convergent 
and preserve as many properties of the original model as possible. We mention that the numerical scheme must be carefully designed to obtain good approximations.
%
{\color{black}
Indeed, the well-known Euler Maruyama (EM) scheme cannot preserve the domain of the original model and the super-linear drift and diffusion coefficients of \eqref{2023SIS-eq:SIS_original} make the resulting approximations divergent \cite{hutzenthaler2011strong}.
In the past few years, many researchers have
made a great of efforts to design implicit \cite{beyn2016stochastic,beyn2017stochastic,higham2002strong,wang2020mean,andersson2017mean}
and explicit \cite{hutzenthaler2012strong,sabanis2013note,wang2013tamed,mao2015truncated,tretyakov2013fundamental,hutzenthaler2020perturbation,brehier2020approximation,sabanis2019explicit,cai2022positivity,yi2021positivity} schemes for SDEs with super-linearly growing coefficients.
However, boundary and dynamics preserving schemes are much less studied.
}
In \cite{alfonsi2013strong,neuenkirch2014first}, the authors combined a Lamperti-type transformation with the backward Euler method to derive 
strongly convergent and domain-preserving schemes
for scalar SDEs. 
More recently, the authors of \cite{LEI2023114758} proposed a positivity-preserving logarithmic transformed truncated Euler-Maruyama method for scalar SDEs and proved its strong and weak convergence rates.
For the model \eqref{2023SIS-eq:SIS_original}, some authors \cite{chen2021first,yang2021first,yang2023strong,LIU2023107258} relied on the Lamperti-type transformation
and applied some explicit schemes to the transformed SDE that has an additive noise. 
Transforming back yields the desired numerical approximations, which preserves the domain and possesses a strong convergence rate of order one {\color{black}in \cite{chen2021first,yang2021first,yang2023strong} and order 1.5 in \cite{LIU2023107258}}. 
In \cite{yang2022numerical,yang2023numerical}, a truncated Wiener process was used to construct one-half order strong
schemes for the model \eqref{2023SIS-eq:SIS_original}, 
where restrictive assumptions were imposed on the step size $h$ to preserve the domain and dynamics behavior
of the original model.
{\color{black}
It is worthwhile to point out that the numerical approximations produced by the above domain-preserving schemes cover the entire domain $(0,N)$, except for that in \cite{yang2023strong}, where a truncation strategy leads to approximations living in a sub-set of $(0,N)$. 
Also, we mention that the dynamics behavior of the numerical approximations was investigated in 
\cite{yang2023strong,yang2022numerical,LIU2023107258,yang2023numerical}, 
but not in \cite{chen2021first,yang2021first}. 
More precisely, the proposed schemes in \cite{yang2023strong,yang2022numerical,LIU2023107258,yang2023numerical} were proven to reproduce extinction
under certain conditions, and the preservation of the persistence property was achieved in \cite{LIU2023107258,yang2022numerical,yang2023numerical}.
A comparison between numerical schemes 
{\color{red}concerning dynamics behavior}
with ours is made in Table \ref{table:comparison} presented below, to provide a more comprehensive overview of the properties of these domain-preserving schemes.

\begin{table}[htp] \label{table:comparison}
    \footnotesize
    \caption{ \centering{ \bf Comparison between numerical schemes for the stochastic SIS model}}
    \label{2023SIS-tab:intro_comparison}
    \centering
    \begin{tabular}
    { m{2cm}<{\centering} m{2cm}<{\centering} m{2cm}<{\centering} m{2.5cm}<{\centering} m{2cm}<{\centering} }
    \toprule 
    Paper  &  Strong convergence rates & Domain preservation  & Extinction preservation & Persistence preservation \\
    \midrule 
    \cite{yang2023strong}  & 1 & Unconditionally &  Assumptions on the step size & {\color{red}NA}\\
    \midrule 
    \cite{LIU2023107258} & 1.5 & Unconditionally & Assumptions on the model parameters & Unconditionally\\
    \midrule 
    \cite{yang2022numerical,yang2023numerical}& 0.5 & Assumptions on the step size  & Assumptions on the step size & Assumptions on the step size\\
    \midrule 
    This paper  & 1 & Unconditionally & Unconditionally & Unconditionally \\
    \bottomrule
\end{tabular}
\end{table}

}

An interesting question thus arises as to whether
one can construct any
boundary and dynamics preserving scheme for the stochastic epidemic model, without any restriction on the step size $h>0$.
This is desirable, particularly in the multi-level Monte Carlo (MLMC) setting \cite{giles2008multilevel}, 
where one is required to use many simulations with large discretization time step size $h>0$. 
It is therefore natural to look for more advanced numerical schemes that capture such properties unconditionally.
The present work 
aims to give a positive answer,
by introducing an explicit, strongly convergent scheme
that preserves both the {\color{green}dynamics behavior} and the boundaries of the analytic solutions for any given step size $h>0$. 
{\color{black}
To this end, for $M \in \N^+$ we construct a uniform mesh $\{t_k = k h\}_{k=0}^M$ over $[0,T]$ with a uniform step size $h = \tfrac{T}{M}$.
}
Our strategy is to combine a logarithm transformation with a corrected Milstein method.
More precisely, by a logarithm transformation $X_t = \log(I_t)$, we obtain the transformed SDE \eqref{2023SIS-eq:SIS_transformed},
{\color{black}
for which corrected Milstein approximations $\{Y_k\}_{k\geq 0}$, given by \eqref{2023SIS-eq:SIS_log_corrected_Mil_scheme}, are designed.
The correction used here guarantees that the approximations $\{Y_k\}_{k\geq 0}$ live in $(-\infty,\log N)$.
Transforming back by $\mathcal{I}_k := \e^{Y_k}$ gives the desired approximation of $I_{t_k}$, which is boundary preserving and strongly convergent with order one (see Theorem \ref{2023SIS-thm:convergence_transformed}, \ref{2023SIS-thm:convergence_oringinal}):
    \begin{equation}
        \E\Big[
            \sup_{k=0,...,M} 
            \big\vert I_{t_k} - \mathcal{I}_k \big\vert^p
        \Big]
        \leq 
        C h^p, \ p \geq 1.
    \end{equation}
}
Throughout this paper the notation $C$ might be slightly abused to denote a generic positive constant depending on $T$ but independent of $h$, and might vary from each time of appearance.
Recall that a Lamperti-type transformation $X_t = \log (\tfrac{ I_t}{N-I_t} )$ was used in
\cite{chen2021first,LIU2023107258,yang2021first,yang2023strong}  to design explicit boundary preserving schemes.
Different from \cite{chen2021first,LIU2023107258,yang2021first,yang2023strong}, we turn to a direct logarithm transformation 
$X_t = \log(I_t)$, which, combined with a corrected Milstein time-stepping scheme \eqref{2023SIS-eq:SIS_log_corrected_Mil_scheme},
helps
{\color{black}covering the entire domain $(0,N)$ and preserving} the dynamics behavior of original model for any step size $h>0$.
Under exactly the same conditions as required in \cite{gray2011stochastic}, the proposed scheme reproduces the extinction and persistence properties of the original model (see Theorem \ref{2023SIS-thm:extinction_scheme}, \ref{2023SIS-thm:persistence_scheme}).
This is also confirmed in the numerical experiments,
showing that the proposed scheme performs well in dynamics preserving, even with large step sizes.

The rest of this paper is organized as follows.
In the next section, the basic properties of the stochastic SIS epidemic model are stated and the numerical scheme is proposed. 
In Section \ref{2023SIS-section:Error_analysis}, the first order strong convergence is elaborated.
Section \ref{2023SIS-section:Dynamics_behavior} is devoted to the analysis of the dynamics behavior for the proposed scheme.
Numerical experiments are presented in Section \ref{2023SIS-section:numerical_experiments} to verify our theoretical findings.
A short conclusion is made finally in Section \ref{2023SIS-section:conclusion}.



\section{The stochastic epidemic model and the proposed scheme}\label{2023SIS-section:SIS_model_and_scheme}

The present section  revisits the stochastic epidemic model
\eqref{2023SIS-eq:SIS_original} and attempts to 
design a desirable time-stepping scheme.
For the well-posedness of \eqref{2023SIS-eq:SIS_original}, we recall the following lemma, which can be found in \cite[Theorem 3.1]{gray2011stochastic}.

\begin{lemma}\label{lem:SIS_solution_property}
For any initial value $I_0 \in (0,N)$, the SDE \eqref{2023SIS-eq:SIS_original} has a unique global positive solution $I_t\in (0,N)$ for all $t \geq 0$ with probability one, namely,
    \begin{equation}
        \mathbb{P} (I_t \in (0,N),\forall t\geq 0)=1.
    \end{equation} 
\end{lemma}

We utilize a logarithmic transform $X_t=\log I_t$ to turn \eqref{2023SIS-eq:SIS_original} into 
\begin{equation}\label{2023SIS-eq:SIS_transformed}
\left\{ 
    \begin{array}{l}
    X_t-X_0=\int_0^t f(X_s)\d s +\int_0^t g(X_s) \d W_s,\ t\in [0,T],\\
    X_0 = \log I_0,
    \end{array} \right.
\end{equation} 
where and throughout the article 
$f,g:\R \rightarrow \R$ are defined as 
\begin{align}
\label{2023SIS-eq:f}
    f(x)
    & :=
    -\tfrac{1}{2}\sigma^2 \e^{2x}
    +
    (\sigma^2 N - \beta) \e^{x}
    +
    (\beta N-\mu-\gamma-\tfrac{1}{2}\sigma^2N^2),\\
\label{2023SIS-eq:g}
    g(x)
    & :=
    \sigma(N - \e^{x}).
\end{align}
One can readily obtain from Lemma \ref{lem:SIS_solution_property} that
\begin{equation}\label{2023SIS-X_solution_existence}
   \mathbb{P} \big(
       X_t \in (-\infty,\log N), \,
       \forall t\geq 0
   \big)
   = 1.
\end{equation} 
Moreover, the coefficients $f$ and $g$ 
satisfy the Lipschitz conditions in the domain $(-\infty,\log N)$, 
which will be used in the error analysis.
\begin{lemma}\label{2023SIS-lem:global_Lipschitz}
    For any $x,y\in(-\infty,\log N)$, there exists a non-negative constant $L$ such that
    \begin{equation}
        \vert f(x)-f(y) \vert 
        \vee 
        \vert g(x)-g(y) \vert
        \leq 
        L \vert x-y \vert
    \end{equation}
    and 
    \begin{equation}
        \vert g(x)g'(x)-g(y)g'(y) \vert 
        \leq 
        L \vert x-y \vert.
    \end{equation}
\end{lemma}
\textbf{Proof:} 
The assertions are trivial due to the boundedness of 
$\e^x$ in $(-\infty, \log N)$.
\qed

For any chosen positive integer $M \in \N^+$, a uniform mesh 
$\{t_n = n h\}_{n=0}^M$
is constructed with the uniform step size $h=T/M$. 
For convenience, we 
define
\begin{align}
    \Delta W_k:=W_{t_{k+1}}-W_{t_{k}},
    \quad 
    \forall k=0,...,M-1.
\end{align} 
In addition, for $s\in[0,T]$, denote 
\begin{align}
    \kappa (s) 
    :=
    \sup\limits_{k=0,...,M}\{t_k:t_k\leq s\}.
\end{align}

Letting $\alpha \in (0, 1]$, $\theta \geq  \frac{3}{2}$ 
be fixed, 
{\color{black}
we propose a corrected Milstein scheme $\{Y_k\}_{k=0,...,M}$ starting from $Y_0=X_0$ for
the transformed SDE
\eqref{2023SIS-eq:SIS_transformed},
} given by
{\color{black}
\begin{equation}\label{2023SIS-eq:SIS_log_corrected_Mil_scheme}
\left\{ 
    \begin{array}{l}
        \bar {Y}_{k+1}
        =
        Y_k
        +
        f(Y_k) h 
        +
        g(Y_k) \Delta W_k
        +
        g(Y_k) g'(Y_k) \Delta \zeta_k,\\
        {Y}_{k+1}
        = 
        \bar {Y}_{k+1}
        \mathbbm{1}_{{\color{red}\{}\bar {Y}_{k+1} < \log N{\color{red}\}}}   
        +
        (
            \log N - \alpha h^{\theta}
        )
        \mathbbm{1}_{{\color{red}\{}\bar {Y}_{k+1} \geq \log N{\color{red}\}}}
        ,
        \
        k = 0,1,...,M-1,
    \end{array} \right.
\end{equation}
}
where 
\begin{equation}
    \Delta \zeta_k
    : =
    \int_{t_k}^{t_{k+1}} 
    \int_{t_k}^s
    \, \d W_r \d W_s
    =
    \half \Delta W_k^2 - \half h
\end{equation}
and {\color{black}$\mathbbm{1}_A$ is the {\color{red} indicator function of some set $A$}, i.e., $\mathbbm{1}_A(\omega)=1$ when $\omega \in A$ and $\mathbbm{1}_A(\omega)=0$ when $\omega \notin A$}.
%
Then transforming the $\{Y_k\}_{k=0}^{M}$ back 
gives the true numerical approximations 
{\color{black}
$\{\mathcal{I}_k\}_{k=0}^{M}$
} for the original model \eqref{2023SIS-eq:SIS_original},
given by
{\color{black}
\begin{equation}\label{2023SIS-eq:SIS_scheme_transforming_back}
    \mathcal{I}_k  = e^{Y_k},
    \quad
    k=0,...,M.
\end{equation}
}
Such a scheme is termed as the logarithmic corrected Milstein (LCM) method.
In light of \eqref{2023SIS-eq:SIS_log_corrected_Mil_scheme}, one knows 
\begin{equation}\label{2023SIS-Y_solution_existence}
    \mathbb{P} \big(
         Y_k \in (-\infty,\log N),\, {\color{green}k=0,...,M}
    \big)
    =
    1
\end{equation} 
and thus
\begin{equation}
    \mathbb{P}
    \big(
         \mathcal{I}_k \in (0, N),\, {\color{green}k=0,...,M}
    \big)
    =
    1.
\end{equation}
In other words, the LCM method is able to preserve the domain
of the stochastic epidemic model.

In the next two lemmas we establish the moment and exponential moment bounds of the solution for the transformed SDE \eqref{2023SIS-eq:SIS_transformed} and its numerical approximation \eqref{2023SIS-eq:SIS_log_corrected_Mil_scheme}.

\begin{lemma}[Exponential moment bounds]\label{2023SIS-lem:exp_moment_bounds}
Let $T > 0$ and 
{\color{black}
$\{X_t\}_{t\in[0,T]}$, $\{Y_k\}_{k=0}^{M}$
}be defined as \eqref{2023SIS-eq:SIS_transformed}, \eqref{2023SIS-eq:SIS_log_corrected_Mil_scheme}.
For any $q \geq 0$, it holds that
    \begin{equation}
        \sup_{t \in [0,T]} 
        \mathbb{E}\big[ 
            {\rm e}^{q X_t}
        \big] 
        \bigvee 
        \sup_{k=0,...,M}
        \mathbb{E}\big[ 
            {\rm e}^{q Y_k}
        \big]
        < 
        \infty.
    \end{equation} 
\end{lemma}
\textbf{Proof:} 
The assertion follows due to \eqref{2023SIS-X_solution_existence} and \eqref{2023SIS-Y_solution_existence}.
\qed

\begin{lemma}[Moment bounds]\label{2023SIS-lem:moment_bounds}
Let $T > 0$ and 
{\color{black}
$\{X_t\}_{t\in[0,T]}$, $\{Y_k\}_{k=0}^{M}$
}
be defined as \eqref{2023SIS-eq:SIS_transformed}, \eqref{2023SIS-eq:SIS_log_corrected_Mil_scheme}.
For any $q \geq 0$, it holds that
    \begin{equation}
        \sup_{t \in [0,T]} 
        \mathbb{E}\big[ 
            \vert X_t \vert ^q
        \big] 
        \bigvee 
        \sup_{k=0,...,M}
        \mathbb{E}\big[ 
             \vert Y_k \vert ^q
        \big]
        < 
        \infty.
    \end{equation} 
\end{lemma}
\textbf{Proof:} 
The assertion can be derived directly by definitions of $\{X_t\}_{t\in[0,T]}$, $\{Y_k\}_{k=0}^{M}$ and Lemma \ref{2023SIS-lem:exp_moment_bounds}.
\qed

    
    In view of Lemmas \ref{2023SIS-lem:global_Lipschitz}, \ref{2023SIS-lem:exp_moment_bounds} and \ref{2023SIS-lem:moment_bounds}, moments of the analitical and numerical solutions
    we encounter later are all bounded, which will not be emphasized further.


\section{The first-order strong convergence of the scheme}
\label{2023SIS-section:Error_analysis}

In this section we attempt to present the analysis of the strong convergence rate for the proposed LCM scheme.
To begin with, we recall first a discrete version of the Burkholder-Davis-Gundy inequality, quoted from \cite[Lemma 4.1]{hutzenthaler2011convergence}.
\begin{lemma}\label{2023SIS-lem:BDG.discrete}
Let 
$M \in \mathbb{N}$ 
and 
$\xi_1,...,\xi_M: \Omega \rightarrow \mathbb{R}$ 
be 
$\mathcal{F}/\mathcal{B}(\mathbb{R})$-measurable mappings with 
$\E \big[ \vert  \xi_n\vert  ^2 \big] < +\infty$ 
for all $n \in \{1,...,M\}$ and with 
$\E \big[ \xi_{n+1} \vert    \xi_1,...,\xi_n \big] = 0$
for all $n \in \{1,...,M-1\}$. 
Then
\begin{equation}
\begin{aligned}
    \|\xi_1+\cdots+\xi_M \|  _{\mathcal{L}^p}
    \leq
    C_p \cdot 
    \Big(
         \|\xi_1\|_{\mathcal{L}^p}^2
         +
         \cdots
         +
         \|\xi_M\|_{\mathcal{L}^p}^2
    \Big)^{\half}
\end{aligned}
\end{equation}
for every $p\in [2,+\infty)$, where $C_p, p \in [2,+\infty)$ are universal constants.
\end{lemma}
Here and below we denote $\|\xi\|_{\mathcal{L}^q}: = \big(\E [\vert  \xi\vert  ^q]\big)^{\tfrac1q} \in [0,+\infty]$ for any $q \in[1, +\infty)$.
Now we are ready to carry out error analysis on the mesh grids.

\begin{theorem}\label{2023SIS-thm:convergence_transformed}
    Let $T > 0$ and $X_t$, $Y_k$ be defined as \eqref{2023SIS-eq:SIS_transformed}, \eqref{2023SIS-eq:SIS_log_corrected_Mil_scheme}.
    For any $q \geq 1$, it holds that
    \begin{equation}
        \E \Big[
            \sup\limits_{k = 0,...,M}
            \vert X_{t_k} - Y_k \vert ^q
        \Big]
        \leq
        C h^q.
    \end{equation}
\end{theorem}
\textbf{Proof:} 
It is easy to deduce that
{\color{black}
\begin{equation}\label{2023SIS-eq:pre_estimate_of_onestep} 
    \begin{aligned}
        &\vert X_{t_{k+1}}-Y_{k+1}\vert^2\\
        & =
        \mathbbm{1}_{{\color{red} \{}
            \bar Y_{k+1}
            < 
            \log N
        {\color{red} \}}}
        \vert X_{t_{k+1}} - \bar Y_{k+1} \vert^2 \\
        & \quad +
        \mathbbm{1}_{\{
            \bar Y_{k+1}
            \geq
            \log N 
        \}} 
        \cdot 
        \mathbbm{1}_{\{
            X_{t_{k+1}}
            \leq 
            \log N - \alpha h^{\theta}
        \}} 
        \vert X_{t_{k+1}} - (\log N - \alpha h^{\theta}) \vert ^2 \\
        & \quad + 
        \mathbbm{1}_{
            \{\bar Y_{k+1} \geq \log N \}
        } 
        \cdot 
        \mathbbm{1}_{\{
            \log N > X_{t_{k+1}} 
            > 
            \log N - \alpha h^{\theta}
        \}} 
        \vert X_{t_{k+1}} - ( \log N - \alpha h^{\theta} ) \vert ^2 \\
        & \leq 
        \vert X_{t_{k+1}} - \bar Y_{k+1}\vert ^2
        +
        \alpha^2 h^{2\theta}.
    \end{aligned}
\end{equation} 
}
Recalling \eqref{2023SIS-eq:SIS_transformed} and \eqref{2023SIS-eq:SIS_log_corrected_Mil_scheme} we have
\begin{equation}\label{2023SIS-eq:one_step_estimate}
\begin{aligned}
        X_{t_{k+1}} - \bar Y_{k+1}
        & =
        X_{t_k}
        -
        Y_k
        + 
        \Big( f(X_{t_k})-f(Y_k) \Big) h
        +
        \Big( g(X_{t_k})-g(Y_k) \Big) \Delta W_k\\
        & \quad +
        \Big (g(X_{t_k}) g'(X_{t_k}) - g(Y_k) g'(Y_k) \Big) \Delta \zeta_k
        +
        R_k,
\end{aligned}
\end{equation} 
where we denote
{\color{black}
\begin{equation}
    \begin{aligned}
        R_k
        & : =
        \int_{t_k}^{t_{k+1}} 
        [f(X_s)-f(X_{t_k})] \, \d s
        +
        \int_{t_k}^{t_{k+1}} 
        [g(X_s) - g(X_{t_k})] \, \d W_s\\
        & \quad
        -
        \int_{t_k}^{t_{k+1}}
        \int_{t_k}^s
        g(X_{t_k}) g'(X_{t_k}) \, \d W_r \d W_s.
    \end{aligned}
\end{equation}
}
Next we estimate $R_k$ first. For $k=0,...,M-1$, 
the It\^{o} formula shows that
{\color{black}
\begin{equation}
    \begin{aligned}
        R_k 
        & =
        \int_{t_k}^{t_{k+1}} 
        \int_{t_k}^s
            \big[f'(X_r) f(X_r) + \half f''(X_r) g^2(X_r)\big]
        \, \d r \d s
        +
        \int_{t_k}^{t_{k+1}} 
        \int_{t_k}^s
            f'(X_r) g(X_r)
        \, \d W_r \d s\\
        & \quad 
        +
        \int_{t_k}^{t_{k+1}} 
        \int_{t_k}^s
            \big[g'(X_r) f(X_r) + \half g''(X_r) g^2(X_r)\big] 
        \, \d r \d W_s\\
        & \quad 
        +
        \int_{t_k}^{t_{k+1}} 
        \int_{t_k}^s
            \big[g(X_r) g'(X_r) - g(X_{t_k}) g'(X_{t_k})\big]
        \, \d W_r \d W_s\\
        & = R_k^{(1)} + R_k^{(2)}
        ,
    \end{aligned}
\end{equation}
}
where we split $R_k$ into two parts as follows:
{\color{black}
\begin{align}
    R_k^{(1)}
    & : =
    \int_{t_k}^{t_{k+1}} 
    \int_{t_k}^s
        \big[f'(X_r) f(X_r) + \half f''(X_r) g^2(X_r)\big]
    \, \d r \d s\\
\nonumber
    R_k^{(2)}
    & : =
    \int_{t_k}^{t_{k+1}} 
    \int_{t_k}^s
        f'(X_r) g(X_r)
    \, \d W_r \d s
    +
    \int_{t_k}^{t_{k+1}} 
    \int_{t_k}^s
        \big[g'(X_r) f(X_r) + \half g''(X_r) g^2(X_r)\big] 
    \, \d r \d W_s\\
    & \quad +
    \int_{t_k}^{t_{k+1}} 
    \int_{t_k}^s
        \big[g(X_r) g'(X_r) - g(X_{t_k}) g'(X_{t_k})\big]
    \, \d W_r \d W_s.
\end{align}
}
The moment inequality and the Jensen inequality together with Lemma \ref{2023SIS-lem:global_Lipschitz} yield that, for any $q \geq 1$
\begin{equation}
    \begin{aligned}
        & \E \Bigg[\bigg \vert
        \int_{t_k}^{t_{k+1}} 
        \int_{t_k}^s
            \big[g(X_r) g'(X_r) - g(X_{t_k}) g'(X_{t_k})\big]
        \, \d W_r \d W_s
        \bigg \vert^{2q} \Bigg]\\
        & \leq
        h^{2q - 2} \cdot
        \int_{t_k}^{t_{k+1}} 
        \int_{t_k}^s
        \E \Big[
            \big \vert
                g(X_r) g'(X_r) - g(X_{t_k}) g'(X_{t_k})
            \big \vert^{2q} 
        \Big]
            \, \d r \d s\\
        & \leq
        L h^{2q - 2} \cdot
        \E \bigg[
            \int_{t_k}^{t_{k+1}} 
            \int_{t_k}^s
            \big\vert
                X_r - X_{t_k}
            \big\vert^{2q} \, \d r \d s
        \bigg]\\
        & \leq
        C L h^{4q - 3}
        \int_{t_k}^{t_{k+1}} 
            \int_{t_k}^s
            \int_{t_k}^r
            \E \Big[ 
                \big\vert
                f(X_u)
                \big\vert^{2q}
            \Big] \,\d u \d r \d s
        \\
        &
        \quad
        +
        C L h^{3q - 3}
        \int_{t_k}^{t_{k+1}} 
        \int_{t_k}^s
        \int_{t_k}^r
        \E \Big[
            \big\vert
            g(X_u)
            \big\vert ^{2q}
        \Big] \,\d u \d r \d s,
    \end{aligned}
\end{equation}
where the H\"{o}lder inequality is also 
used in the last inequality.
Therefore, the boundedness of moments ensure that, 
for any $q \geq 1$ and $k = 0,1,...,M-1$
\begin{align}\label{2023SIS-eq:Rn_estimate}
    \E \Big[
        \big\vert R_k^{(1)} \big\vert ^{2q}
    \Big]
    \leq
    C h^{4q},
    \quad
    \E \Big[
        \big\vert  R_k^{(2)} \big\vert ^{2q}
    \Big]
    \leq
    C h^{3q}
\end{align}
and thus
\begin{equation}
\E \Big[
        \big\vert R_k \big\vert ^{2q}
    \Big]
    \leq
    C h^{3q}.
\end{equation}
Denote for brevity
\begin{align*} 
    e_{k}
    & :=
    X_{t_k} - Y_k,
    & \quad 
    A_k 
    & := 
    \Big( f(X_{t_k}) - f(Y_k) \Big) h,\\ 
    B_k 
    & := 
    \Big( g(X_{t_k}) - g(Y_k) \Big) \Delta W_k,
    & \quad
    D_k 
    & := 
    \Big( g(X_{t_k}) g'(X_{t_k}) - g(Y_k) g'(Y_k) \Big) \Delta \zeta_k.
\end{align*}
Thanks to Lemma \ref{2023SIS-lem:global_Lipschitz}, we have
\begin{align}\label{2023SIS-eq:ABC_bounds}
    \vert A_k \vert
    \leq
    L h \vert e_k \vert,
    \quad
    \vert B_k \vert
    \leq
    L \vert e_k \vert \vert \Delta W_k \vert,
    \quad
    \vert D_k \vert
    \leq
    L \vert e_k \vert \vert \Delta \zeta_k \vert.
\end{align}
Plugging \eqref{2023SIS-eq:one_step_estimate} into \eqref{2023SIS-eq:pre_estimate_of_onestep} and by the Young inequality we get, for any $k=0,...,M-1$,
\begin{equation}\label{2023SIS-eq:pre_square_estimate}
    \begin{aligned}
        e^2_{k+1} 
        & \leq 
        e^2_k + A^2_k + B^2_k + D^2_k + R^2_k 
        + 
        2 e_k A_k + 2 e_k B_k + 2 e_k D_k + 2 e_k R_k 
        \\
        & \quad 
        + 
        2 A_k B_k + 2 A_k R_k + 2 A_k D_k
        + 
        2 B_k D_k + 2 B_k R_k
        + 
        2 D_k R_k
        +
        \alpha^2 h^{2\theta}\\
        & \leq
        (1 + 3 h)e^2_k 
        + 
        \big(4 + h^{-1} \big)
        (A^2_k+D^2_k)
        +
        4 \big( B^2_k + R^2_k \big)
        +
        2 e_k B_k
        \\
        & \quad 
        +
        h^{-1}
        \vert R^{(1)}_k \vert ^2
        +
        2 e_k R^{(2)}_k
        +
        \alpha^2 h^{2\theta}.
    \end{aligned}
\end{equation}
By iteration, we arrive at 
\begin{equation}\label{2023SIS-eq:ek^2_iteration}
    \begin{aligned}
        e^2_{k+1}
        & \leq
        (1 + 3 h)^k
        \sum_{i = 0}^k
        \tfrac{1}{(1 + 3 h)^i}
        \Big[
            (4 + h^{-1})
            (A^2_i+D^2_i) 
            +
            4 \big( B^2_i + R^2_i \big)
            +
            h^{-1}
            \vert R^{(1)}_i \vert ^2
        \Big]\\
        & \quad 
        +
        (1 + 3 h)^k
        \sum_{i = 0}^k
        \tfrac{1}{(1 + 3 h)^i}
        \Big[
            2 e_i B_i
            +
            2 e_i R^{(2)}_i
        \Big]
        +
        C h^{2 \theta - 1}.
    \end{aligned}
\end{equation}
For any $n=0,...,M-1$ and $p \geq 2$, raising both sides of \eqref{2023SIS-eq:ek^2_iteration} to the power of $p$, taking supremum and expectation as well as applying the Jensen inequality we get 
\begin{equation} \label{2023SIS-eq:ek^2p_estimate}
    \begin{aligned}
        \E \Big[
            \sup_{k=0,...,n}
            \vert e_{k+1} \vert ^{2p}
        \Big] 
        &\leq 
        C h
        \sum_{i = 0}^n
        \E \Big[ \vert e_i \vert^{2p}  \Big]
        +
        C
        (1 + 3 h)^{np}\cdot
        \E \Bigg[
        \sup_{k=0,...,n}
        \bigg \vert
            \sum_{i = 0}^k
            \tfrac{1}{(1 + 3 h)^i}
            2 e_i B_i
        \bigg \vert ^p \Bigg]\\
        & \quad 
        +    
        C
        (1 + 3 h)^{np}\cdot
        \E \Bigg[
        \sup_{k=0,...,n}
        \bigg \vert
            \sum_{i = 0}^k
            \tfrac{1}{(1 + 3 h)^i}
            2 e_i R^{(2)}_i
        \bigg \vert ^p \Bigg]
        +
        C h^{2p}.
    \end{aligned}
\end{equation}

By denoting 
    $\xi_{i}^{(1)}
    :=
    ( 1 + 3 h )^{-i} 
    e_i B_i
    $
and
    $\xi_{i}^{(2)}
    :=
    ( 1 + 3 h )^{-i} 
    e_i R_{i}^{(2)},
    $ 
it is not difficult to check that 
    $
    \E \big[ 
        \xi_{i+1}^{(1)} 
        \vert   
        \xi_1^{(1)},...,\xi_i^{(1)} 
    \big] = 0
    $
and
    $
    \E \big[ 
        \xi_{i+1}^{(2)} 
        \vert   
        \xi_1^{(2)},...,\xi_i^{(2)} 
    \big] = 0.
    $
Moreover, 
    $
    \{ \Xi_k^{(1)} \}_{k=1}^{M}: =
    \Big\{ 
    \sum\limits_{i=0}^{k-1} \xi_i^{(1)}
    \Big\}_{k=1}^{M}
    $
    and
    $
    \{ \Xi_k^{(2)} \}_{k=1}^{M}: =
    \Big\{ 
    \sum\limits_{i=0}^{k-1} \xi_i^{(2)}
    \Big\}_{k=1}^{M}
    $
are square-integrable discrete time martingales.
Hence the Doob discrete martingale inequality, Lemma \ref{2023SIS-lem:BDG.discrete} and the Jensen inequality can be applied to infer that
\begin{equation}\label{2023SIS-eq:ei_Bi_estimate}
    \begin{aligned}
        \E \Bigg[\bigg \vert
            \sup_{k=0,...,n}
            \sum_{i = 0}^k
            \tfrac{1}{(1 + 3 h)^i}
            2 e_i B_i
        \bigg \vert ^p \Bigg]
        & \leq
        C
        \cdot
        (\tfrac1h)^{\frac{p}{2}-1}
        \sum_{i = 0}^n
        \E \bigg[\Big \vert
             2e_i B_i
        \Big \vert ^p \bigg]
        \leq
        C h 
        \sum_{i = 0}^n
        \E \Big[ \vert e_i \vert^{2p} \Big].
    \end{aligned}
\end{equation}
With the help of the Young inequality additionally, we have
\begin{equation}\label{2023SIS-eq:ei_Ri2_estimate}
    \begin{aligned}
        \E \Bigg[\bigg \vert
            \sup_{k=0,...,n}
            \sum_{i = 0}^k
            \tfrac{1}{(1 + 3 h)^i}
            2 e_i R_{i}^{(2)}
        \bigg \vert ^p \Bigg]
        & \leq
        C
        \cdot
        (\tfrac1h)^{\frac{p}{2}-1}
        \sum_{i = 0}^n
        \E \bigg[\Big \vert
            2 e_i R_{i}^{(2)}
        \Big \vert ^p \bigg]\\
        & \leq
        C h 
        \sum_{i = 0}^n
        \E \Big[ \vert e_i \vert^{2p} \Big]
        +
        C h^{2p}.
    \end{aligned}
\end{equation}
Inserting \eqref{2023SIS-eq:ei_Bi_estimate}-\eqref{2023SIS-eq:ei_Ri2_estimate} into \eqref{2023SIS-eq:ek^2p_estimate} we arrive at
\begin{equation}
    \begin{aligned}
        \E \Big[
            \sup_{k=0,...,n} \vert e_{k+1} \vert^{2p}
        \Big] 
        \leq  
        C h
        \sum_{i=1}^n 
        \E \Big[ 
            \vert e_i\vert^{2p}
        \Big] 
        +
        C h^{2p}
        < \infty.
    \end{aligned}
\end{equation}
Finally, the Gronwall inequality and the Lyapunov inequality finish the proof.
\qed

Equipped with the above convergence result
for the approximations of the transformed SDE, 
we therefore obtain a convergence rate of order one 
for the LCM scheme.
\begin{theorem}\label{2023SIS-thm:convergence_oringinal}
Let $T > 0$ and $I_t$, $\mathcal{I}_k$ be defined as \eqref{2023SIS-eq:SIS_original}, \eqref{2023SIS-eq:SIS_scheme_transforming_back}. 
For any $q \geq 1$, there exists some positive constant $C$ independent of $h$ such that  
    \begin{equation}
        \E\Big[
            \sup_{k=0,...,M} 
            \big\vert I_{t_k} - \mathcal{I}_k \big\vert^q
        \Big]
        \leq 
        C h^q.
    \end{equation}
\end{theorem}
\textbf{Proof:}  
The H{\"o}lder inequality and boundedness of 
the exponential moment together imply that for $p \geq 2$,
\begin{equation}
        \begin{aligned}
            \E\Big[ 
                \sup_{k=0,...,M}
                \big\vert e^{X_{t_k}} - e^{Y_k} \big\vert^p
            \Big] 
            \leq  
            C \E\Big[ 
                \sup_{k=0,...,M} 
                \big\vert X_{t_k} - Y_k \big\vert^p
            \Big]
            \leq 
            C h^p.
        \end{aligned}
    \end{equation}
    Thus by the Lyapunov inequality the assertion holds for any $q \geq 1$.
\qed

\section{Dynamics behavior of numerical approximations}
\label{2023SIS-section:Dynamics_behavior}

In this section we investigate the ability of the proposed
scheme to reproduce significant dynamics behavior of
the stochastic epidemic model.
More precisely, we examine the extinction and persistence properties of the time discretization and show that,
under exactly the same conditions as required in \cite{gray2011stochastic},
the proposed scheme is able to preserve the dynamics behavior of the original model for any step size $h>0$.

\subsection{Unconditional extinction preserving}
As the first {\color{green}dynamics behavior} of the stochastic model, 
we focus on the extinction property that is defined as follows.
\begin{definition}[Extinction]\label{2023SIS-def:extinction}
If the unique positive solution $I_t$, $t \geq 0$ of the stochastic SIS model tends to zero in almost sure sense as $t \rightarrow \infty$, that is, 
\begin{equation}
    \lim\limits_{t \rightarrow \infty} I_t = 0 \quad a.s.,
\end{equation}
then we say that the infected population system has the extinction property.
Similarly, the numerical approximation $\{\mathcal{I}_k\}_{k\geq 0}$
is said to have the extinction property if
\begin{equation}
    \lim\limits_{k h \rightarrow \infty} \mathcal{I}_k = 0
    \quad
    a.s..
\end{equation}
\end{definition}

The next theorem presents some sufficient conditions
to ensure the extinction property of the original stochastic SIS model \eqref{2023SIS-eq:SIS_original}, 
which is indeed quote from \cite[Theorems 4.1, 4.3]{gray2011stochastic}.
\begin{theorem}\label{2023SIS-thm:extinction_SIS_original}
Let the basic reproduction number of the stochastic SIS
model \eqref{2023SIS-eq:SIS_original} be less than one, i.e.,
\begin{align}\label{2023SIS-eq:extinction_1_R0S}
    R_0^S = \tfrac{\beta N}{\mu + \gamma} - \tfrac{\sigma^2 N^2}{2(\mu + \gamma)}
    < 1.
\end{align}
For any given initial value $I_0 \in (0,N)$, we have
the following extinction properties.
\begin{enumerate}[(i)]
    \item If $\sigma^2 \leq \frac{\beta}{N}$, then
    \begin{align}\label{2023SIS-eq:extinction_1}
        \limsup\limits_{t \rightarrow \infty}
        \frac{\log{I_t}}{t}
        \leq
        \beta N - \half \sigma^2 N^2 - \mu - \gamma 
        < 0
        \quad
        a.s..
    \end{align}
    \item If $\sigma^2 > \max \Big\{ \frac{\beta}{N}, \, \frac{\beta^2}{2(\mu + \gamma)} \Big\}$, then
    \begin{align}\label{2023SIS-eq:extinction_2}
        \limsup\limits_{t \rightarrow \infty}
        \frac{\log{I_t}}{t}
        \leq
        \tfrac{\beta^2}{2 \sigma^2} - \mu - \gamma
        < 0
        \quad
        a.s..
    \end{align} 
\end{enumerate}
\end{theorem}
We mention that the assertions \eqref{2023SIS-eq:extinction_1} and \eqref{2023SIS-eq:extinction_2} 
imply that the solution $I_t$ of the stochastic SIS model tends to zero exponentially almost surely as time grows. In other words, the disease dies out with probability one,
as time tends to infinity.
In what follows, we investigate under which conditions 
the numerical approximation will admit the extinction property in the sense that
$\mathcal{I}_k$ tends to zero exponentially almost surely, which is to be stated in the next theorem.

\begin{theorem}\label{2023SIS-thm:extinction_scheme}

Let the condition \eqref{2023SIS-eq:extinction_1_R0S} hold, i.e., $R_0^S < 1$.
For any given initial value $I_0 \in (0,N)$ and 
for any step size $h>0$, 
the numerical approximations $\mathcal{I}_k$ possess
the following extinction properties.
\begin{enumerate}[(i)]
    \item If $\sigma^2 \leq \frac{\beta}{N}$, then
    \begin{align}
        \limsup\limits_{kh \rightarrow \infty}
        \frac{\log{\mathcal{I}_k}}{kh}
        \leq
        \beta N - \half \sigma^2 N^2 - \mu - \gamma 
        < 0
        \quad
        a.s..
    \end{align}
    \item If $\sigma^2 > \max \Big\{ \frac{\beta}{N}, \, \frac{\beta^2}{2(\mu + \gamma)} \Big\}$, then
    \begin{align}
        \limsup\limits_{kh \rightarrow \infty}
        \frac{\log{\mathcal{I}_k}}{kh}
        \leq
        \tfrac{\beta^2}{2 \sigma^2} - \mu - \gamma
        < 0
        \quad
        a.s..
    \end{align} 
In other words, the numerical approximations produced by \eqref{2023SIS-eq:SIS_log_corrected_Mil_scheme} tends to zero exponentially almost surely as time evolves.
\end{enumerate}

\end{theorem}

\textbf{Proof:}
    Recall first that
    \begin{align}
        \nonumber
        f(x)
        & =
        -\tfrac{1}{2}\sigma^2 \e^{2x}
        +
        (\sigma^2 N - \beta) \e^{x}
        +
        (\beta N-\mu-\gamma-\tfrac{1}{2}\sigma^2N^2),
        \quad
        x \in (-\infty, \log N).
    \end{align}
    In the first case that $\sigma^2 \leq \tfrac{\beta}{N}$, clearly
    $f(x) \leq \beta N-\mu-\gamma-\frac{1}{2}\sigma^2N^2 < 0$.
    In the second case that
    $\sigma^2 > \max \Big\{ \frac{\beta}{N}, \, \frac{\beta^2}{2(\mu + \gamma)} \Big\}$,  
    $f(x)$ reaches its maximum value 
    $ \tfrac{\beta^2}{2 \sigma^2} -\mu - \gamma < 0 $
    when 
    $ e^x = \tfrac{\sigma^2 N - \beta}{\sigma^2}$.
    By denoting
    \begin{align}
        f_{max}^{\sigma}
        : =
        \begin{cases}
            \beta N-\mu-\gamma-\frac{1}{2}\sigma^2N^2,
            \quad 
            \sigma^2 \leq \tfrac{\beta}{N};
            \\
            \tfrac{\beta^2}{2 \sigma^2} -\mu - \gamma,
            \quad
            \sigma^2 > \max \Big\{ \frac{\beta}{N}, \, \frac{\beta^2}{2(\mu + \gamma)} \Big\},
        \end{cases}
    \end{align}
in any case we have
    \begin{align}
        \label{2023SIS-eq:extinction_1_f_bounds}
        f(x)
        \leq 
        f_{max}^{\sigma}
        < 0.
    \end{align}
    Recall that
{\color{black}
\begin{equation*}
\left\{ 
    \begin{array}{l}
        \bar {Y}_{k+1}
        =
        Y_k
        +
        f(Y_k) h 
        +
        g(Y_k) \Delta W_k
        +
        g(Y_k) g'(Y_k) \Delta \zeta_k,\\
        {Y}_{k+1}
        = 
        \bar {Y}_{k+1}
        \mathbbm{1}_{{\color{red} \{} \bar {Y}_{k+1} < \log N {\color{red} \}}}   
        +
        (
            \log N - \alpha h^{\theta}
        )
        \mathbbm{1}_{\{ \bar {Y}_{k+1} \geq \log N \}}
        .
    \end{array} \right.
\end{equation*} 
}
    It follows from \eqref{2023SIS-eq:extinction_1_f_bounds} and by induction that
    \begin{align}
        \nonumber
        \bar Y_{k+1}
        & \leq 
        Y_k
        +
        h f_{max}^{\sigma}
        +
        g(Y_k) \Delta W_k
        +
        g(Y_k) g'(Y_k) \Delta \zeta_k\\
        & \leq 
        Y_0
        +
        (k+1) h f_{max}^{\sigma}
        +
        \sum_{i=0}^{k}
        g(Y_i) \Delta W_i
        +
        \sum_{i=0}^{k}
        g(Y_i) g'(Y_i) \Delta \zeta_i,
    \end{align}
    and thus
    \begin{align}
        Y_{k+1}
        \leq
        \bar Y_{k+1}
        \leq
        Y_0
        +
        (k+1) h f_{max}^{\sigma}
        +
        \sum_{i=0}^{k}
        g(Y_i) \Delta W_i
        +
        \sum_{i=0}^{k}
        g(Y_i) g'(Y_i) \Delta \zeta_i.     
    \end{align}
    Therefore,
    \begin{align}
    \label{eq:extinction_Y_induction}
        \nonumber
        \limsup\limits_{k h \rightarrow \infty}
        \frac{Y_{k}}{k h}
        & \leq
        f_{max}^{\sigma}
        +
        \limsup\limits_{k h \rightarrow \infty}
        \frac{1}{k h}
        \sum_{i=0}^{k-1}
        g(Y_i) \Delta W_i
        +
        \limsup\limits_{k h \rightarrow \infty}
        \frac{1}{k h}
        \sum_{i=0}^{k-1}
        g(Y_i) g'(Y_i) \Delta \zeta_i
        \quad
        a.s..
    \end{align}
    Further, note that 
    $$
    \Xi^{(3)}_t
        : =
            \int_0^t g(Y_{\kappa(s)}) {\rm d}W_s,\ 
    \Xi^{(4)}_t
        : =
            \int_0^t g'(Y_{\kappa(s)}) g(Y_{\kappa(s)})(W_s-W_{\kappa(s)}) {\rm d}W_s
    $$
    are both square-integrable continuous time martingales.
    By the large number theorem (see e.g., \cite[Theorem 3.4.]{mao2007stochastic}) for martingales we have
    $$
    \lim\limits_{t \rightarrow \infty} 
    \frac{ \Xi^{(3)}_t}{t} 
    =
    \lim\limits_{t \rightarrow \infty} 
    \frac{ \Xi^{(4)}_t}{t}=0,\ a.s.,
    $$
    which yields that 
    
    \begin{align}
        \limsup\limits_{k h \rightarrow \infty}
        \frac{1}{k h}
        \sum_{i=0}^{k-1}
        g(Y_i) \Delta W_i
        =
        \limsup\limits_{k h \rightarrow \infty}
        \frac{1}{k h}
        \sum_{i=0}^{k-1}
        g(Y_i) g'(Y_i) \Delta \zeta_i
        =
        0
        \quad
        a.s..
    \end{align}
    Therefore
    \begin{align}
        \limsup\limits_{k h \rightarrow \infty}
        \frac{Y_{k}}{k h}
        \leq
        f_{max}^{\sigma}
        < 0
        \quad
        a.s..
    \end{align}
    Noting that 
    $
    \log \mathcal{I}_k 
    = 
    Y_k,
    $
    we finally arrive at
    \begin{align}
        \nonumber
        \limsup\limits_{k h \rightarrow \infty}
        \frac{\log \mathcal{I}_k}{k h}
        & =
        \limsup\limits_{k h \rightarrow \infty}
        \frac{ Y_k }{k h}
        \leq
        f_{max}^{\sigma}
        < 0
        \quad
        a.s..
    \end{align}
    The proof is thus completed.
\qed

Theorem \ref{2023SIS-thm:extinction_scheme} shows that the numerical approximation produced by the LCM scheme is able to reproduce the extinction property of the original model for any step size $h>0$.

\subsection{Unconditional persistence preserving}
In this subsection, let us turn to 
the persistence property of the model, 
defined as follows.
\begin{definition}[Persistence]\label{2023SIS-def:persistence}
If the unique positive solution $I_t$, $t \geq 0$ of the stochastic SIS model obeys
\begin{equation}
    \limsup\limits_{t \rightarrow \infty} I_t \geq \lambda \quad a.s.
    \quad
    \text{and}
    \quad
    \liminf\limits_{t \rightarrow \infty} I_t \leq \lambda \quad a.s.,
\end{equation}
where $\lambda$ is a positive constant,
namely, $I_t$ will rise to or above $\lambda$ infinitely often with probability one,
we say that the infected population system has the persistence property.
Similarly, the numerical approximation $\{\mathcal{I}_k\}_{k \geq 0}$ is said to have the persistence property if \begin{equation}
    \limsup\limits_{k h \rightarrow \infty} \mathcal{I}_k \geq \lambda \quad a.s.
    \quad
    \text{and}
    \quad
    \liminf\limits_{k h \rightarrow \infty} \mathcal{I}_k \leq \lambda \quad a.s..
\end{equation}
\end{definition}
{\color{black}
We mention that {\color{green}$\limsup\limits_{t\rightarrow  \infty} I_t$} $\geq  \lambda>0 $ implies that $I_t$ will rise to or above $\lambda$ infinitely often, which means the disease will persist and never die out. 
The other assertion 
     {\color{green}$\liminf\limits_{t\rightarrow  \infty} I_t$} $\leq  \lambda $,
     i.e., $I_t$ will be below $\lambda$ infinitely often,
    then provides more precise information about the range of the "Persistence".
}
We are now in a position to  provide some sufficient conditions to ensure the persistence property of the
original stochastic SIS mode 
\eqref{2023SIS-eq:SIS_original},  which have been
established in \cite[Theorem 5.1]{gray2011stochastic}.
\begin{theorem}\label{2023SIS-thm:persistence_original}
    If
    \begin{align}\label{2023SIS-eq:persistence_condition}
        R_0^S
        =
        \tfrac{\beta N}{\mu + \gamma} - \tfrac{\sigma^2 N^2}{2(\mu + \gamma)}
        > 1,
    \end{align}
    then for any given initial value $I_0 \in (0,N)$, the solution of the stochastic SIS model \eqref{2023SIS-eq:SIS_original} obeys
    \begin{equation}
        \limsup\limits_{t \rightarrow \infty} I_t \geq \lambda \quad a.s. 
        \quad
        \text{and}
        \quad
        \liminf\limits_{t \rightarrow \infty} I_t \leq \lambda \quad a.s.,
    \end{equation}
    where $\lambda$ is the unique root in $(0,N)$ of
    \begin{align}
        \beta N  - \mu - \gamma - \beta \lambda 
        - 
        \half \sigma^2 ( N - \lambda )^2
        = 0.
    \end{align}
    In other words, $I_t$ will rise to or above $\lambda$ infinitely often with probability one.
\end{theorem}
Next, let us investigate the ability of the numerical approximation produced by the LCM scheme to preserve 
the persistence property in the sense that
$\mathcal{I}_k$ rise to or above a positive constant infinitely often with probability one.
\begin{theorem}\label{2023SIS-thm:persistence_scheme}
Let the condition \eqref{2023SIS-eq:persistence_condition} hold, i.e., $R_0^S > 1$.
For any step size $h >0$, if we  chose a proper constant $\alpha\in(0,1]$ in the scheme \eqref{2023SIS-eq:SIS_log_corrected_Mil_scheme} satisfying
    \begin{equation} \label{eq:thm-persistence-alpha-condition}
      \alpha 
        < 
      h^{-\theta}\log \Big(
            \tfrac{\sigma^2N}
                {\sigma^2N-\beta+\sqrt{\beta^2-2\sigma^2(\mu+\gamma)}}
      \Big),
  \end{equation}
then for any given initial value $I_0 \in (0,N)$ the numerical approximation obeys:
    \begin{equation}\label{2023SIS-eq:persistence_I}
        \limsup\limits_{k h \rightarrow \infty} \mathcal{I}_k 
        \geq 
        \lambda \quad a.s. 
        \quad
        \text{and}
        \quad
        \liminf\limits_{k h \rightarrow \infty} \mathcal{I}_k 
        \leq
        \lambda \quad a.s.,
    \end{equation}
    where $\lambda$ is the unique root of
    \begin{align}
        \beta N - \mu - \gamma - \beta  \lambda - \half \sigma^2 ( N - \lambda )^2 = 0.
    \end{align}
    In other words, the numerical approximations produced by \eqref{2023SIS-eq:SIS_log_corrected_Mil_scheme} will rise to or above $\lambda$ infinitely often with probability one.
\end{theorem}
\textbf{Proof:}
    By the assumption \eqref{2023SIS-eq:persistence_condition}, there exists a positive constant $\delta$ such that 
    $
    \delta < R_0^S - 1.
    $
    Recall that
    \begin{equation*}
        f(x)
        =
        -\tfrac{1}{2}\sigma^2 \e^{2x}
        +
        (\sigma^2 N - \beta) \e^{x}
        +
        (\beta N-\mu-\gamma-\tfrac{1}{2}\sigma^2N^2).
    \end{equation*}
    Clearly we have
    \begin{align}
        \label{eq:extinction_F_-infty}
        \lim\limits_{x \rightarrow - \infty}
        f(x)
        & =
        \beta N-\mu-\gamma-\tfrac{1}{2}\sigma^2N^2
        >
        \delta (\mu + \gamma)> 0,\\
        \lim\limits_{x \rightarrow (\log N - \alpha h^{\theta})^-}
        f(x)
        & =
        -\tfrac12 \sigma^2 N^2 \big( 1 - \e^{- \alpha h^{\theta}} \big)^2
        +
        \beta N \big( 1 - \e^{- \alpha h^{\theta}} \big)
        -
        \mu - \gamma.
    \end{align}
    It is easy to derive that 
 \begin{equation}
    \lim\limits_{x \rightarrow (\log N - \alpha h^{\theta})^-}
    f(x)< 0
 \end{equation}
holds true,  provided
 \begin{equation} 
    1-\e^{-\alpha h^{\theta}} < 
    \tfrac{\beta-\sqrt{\beta^2-2\sigma^2(\mu+\gamma)}}{\sigma^2N}.
 \end{equation}
This can be fulfilled under the assumption 
\eqref{eq:thm-persistence-alpha-condition}, that is,
  \begin{equation*}
      \alpha 
        < 
      h^{-\theta}\log \Big(
            \tfrac{\sigma^2N}
                {\sigma^2N-\beta+\sqrt{\beta^2-2\sigma^2(\mu+\gamma)}}
      \Big).
  \end{equation*}
  Noting
  \begin{equation*}
      f'(x)
      =
      -\sigma^2 \e^{2 x}
      +
      (\sigma^2 N - \beta) \e^x,
  \end{equation*}
  one can readily infer that $f'(x) < 0$ 
  and $f(x)$ is decreasing in the case 
  $\sigma^2 N - \beta \leq 0$.
  In the other case $\sigma^2 N - \beta >0$,
  \begin{itemize}
      \item $f(x)$ is strictly increasing as $x < \log(\tfrac{\sigma^2 N - \beta}{\sigma^2}) $,
      \item $f(x)$ is strictly decreasing as $x > \log(\tfrac{\sigma^2 N - \beta}{\sigma^2})$.
  \end{itemize}
    Thus the equation $f(x)=0$ has a unique root, denoted as $\eta$, satisfying $\e^\eta=\lambda$ and $\eta < \log N - \alpha h^{\theta}$.
    
    Now we validate the two assertions in \eqref{2023SIS-eq:persistence_I} in two steps. 
    If the first assertion of \eqref{2023SIS-eq:persistence_I} is not true, then there exists a positive constant $\epsilon$ such that
    \begin{align}
        P(\Omega_1) > \epsilon,
    \end{align}
    where $\Omega_1 := \{ \omega : \limsup\limits_{k h \rightarrow \infty} Y_k(\omega) \leq \eta - 2 \epsilon \}$.
    Hence $\forall \omega \in \Omega_1$, there exists a $m = m(\omega) > 0$ such that
    \begin{align}\label{eq:persistence_contradiction_1}
        \bar Y_k(\omega) = Y_k(\omega) 
        \leq 
        \eta - \epsilon 
        < 
        \log N - \alpha h^{\theta}
    \end{align}
    whenever $k > m$.
    We can choose $\epsilon$ sufficiently small such that $0 < f(\eta - \epsilon)< \delta (\mu + \gamma)$, which, together with \eqref{eq:extinction_F_-infty}, implies that
    \begin{align}\label{2023SIS-eq:persistence_f_bound}
        f(Y_k(\omega)) 
        \geq
        f(\eta - \epsilon),
    \end{align}
    whenever $k > m$.
    Moreover, by the large number theorem of martingales, there is an $\Omega_2 \subset \Omega$ with $P(\Omega_2) = 1$ such that $\forall \omega \in \Omega_2$
    \begin{align}\label{eq:LNT_persistence}
        \lim\limits_{k h \rightarrow \infty}
        \frac{1}{k h}
        \sum_{i=0}^{k-1}
        g(Y_i(\omega)) \Delta W_i(\omega)
        +
        \lim\limits_{k h \rightarrow \infty}
        \frac{1}{k h}
        \sum_{i=0}^{k-1}
        g(Y_i(\omega))g'(Y_i(\omega)) \Delta \zeta_i(\omega)
        =
        0.
    \end{align}
    Fix any $\omega \in \Omega_1 \cap \Omega_2$ and let $k>m+1$. 
    Then it follows from \eqref{2023SIS-eq:persistence_f_bound} and by induction that
    \begin{align}
        \nonumber
        \bar Y_{k+1}(\omega)
        & =
        \bar Y_k(\omega)
        +
        f(Y_k(\omega)) h 
        +
        g(Y_k(\omega)) \Delta W_k(\omega)
        +
        g(Y_k(\omega)) g'(Y_k(\omega)) \Delta \zeta_k(\omega)\\
        \nonumber
            & = Y_{m+1}(\omega)
                +
                \sum_{i=m+1}^{k} f(Y_i(\omega)) h\\
        \nonumber
        & \quad
                +
                \sum_{i=m+1}^{k}
                \Big[
                g(Y_i(\omega)) \Delta W_i(\omega)
                    +
                    g(Y_i(\omega))g'(Y_i(\omega)) \Delta \zeta_i(\omega)
                \Big]\\
        \nonumber
        & \geq
        Y_{m+1}(\omega)
        +
        (k - m) f(\eta - \epsilon) h\\
        & \quad +
        \sum_{i=m+1}^{k}
        \Big[
        g(Y_i(\omega)) \Delta W_i(\omega)
        +
        g(Y_i(\omega))g'(Y_i(\omega)) \Delta \zeta_i(\omega)
        \Big].
    \end{align}
    Thus
    \begin{align}
        \liminf\limits_{k h \rightarrow \infty}
        \frac{\bar Y_k(\omega)}{k h}
        \geq
        f(\eta - \epsilon) > 0,
    \end{align}
    {\color{black}
    and consequently
    \begin{align}
        \lim\limits_{k h \rightarrow \infty}
        \bar Y_k(\omega) = + \infty.
    \end{align}
    which contradicts \eqref{eq:persistence_contradiction_1}.
    }
    Thus
    \begin{align}
        \limsup\limits_{k h \rightarrow \infty} Y_k \geq \eta \quad a.s. .
    \end{align}
%
    Further, if the second assertion of \eqref{2023SIS-eq:persistence_I} is not true, then there exists a positive constant $\epsilon$ such that
    \begin{align}
        P(\Omega_3) > \epsilon,\
    \eta + 2 \epsilon < \log N,
    \end{align}
    where $\Omega_3 := \{ \omega : \liminf\limits_{k h \rightarrow \infty} Y_k(\omega) \geq \eta + 2 \epsilon \}$.
    Hence $\forall \omega \in \Omega_3$, there exists a $m = m(\omega) > 0$ such that
    \begin{align}\label{eq:persistence_contradiction_2}
        Y_k(\omega) \geq \eta + \epsilon
    \end{align}
    whenever $k > m$.
    Simple analysis on $f(x)$ implies that
    \begin{align}\label{eq:F_omega_leq}
        f(Y_k(\omega)) \leq f(\eta + \epsilon)
        <0
    \end{align}
    whenever $k > m$.
    Fix any $\omega \in \Omega_3 \cap \Omega_2$. 
    Then it follows from \eqref{eq:F_omega_leq} and by induction that
    \begin{align}
        \nonumber
        \bar Y_{k+1}(\omega)
        & =
        Y_k(\omega)
        +
        f(Y_k(\omega)) h 
        +
        g(Y_k(\omega)) \Delta W_k(\omega)
        +
        g(Y_k(\omega)) g'(Y_k(\omega)) \Delta \zeta_k(\omega)\\
        \nonumber
        & \leq
        Y_0(\omega)
        +
        \sum_{i=0}^{m}
        f(Y_i(\omega)) h
        +
        (k - m) f(\eta + \epsilon) h\\
        & \quad +
        \sum_{i=0}^{k}
        \Big[
        g(Y_i(\omega)) \Delta W_i(\omega)
        +
        g(Y_i(\omega))g'(Y_i(\omega)) \Delta \zeta_i(\omega)
        \Big].
    \end{align}
    Thus by \eqref{eq:LNT_persistence} we have
    \begin{align}
        \limsup\limits_{k h \rightarrow \infty}
        \frac{\bar Y_k(\omega)}{k h}
        \leq
        f(\eta + \epsilon) < 0,
    \end{align}
    and consequently
    \begin{align}
        \lim\limits_{k h \rightarrow \infty}
        \bar Y_k(\omega) = -\infty.
    \end{align}
    {\color{black}
    Therefore we have
    \begin{align}
        \lim\limits_{k h \rightarrow \infty}
        Y_k(\omega) 
        \leq 
        \lim\limits_{k h \rightarrow \infty}
        \bar Y_k(\omega) 
        =
        -\infty,
    \end{align}
    which contradicts \eqref{eq:persistence_contradiction_2}. 
    Thus 
    \begin{align}
        \liminf\limits_{k h \rightarrow \infty} Y_k \leq \eta \quad a.s..
    \end{align}
    The proof of assertion 
    \eqref{2023SIS-eq:persistence_I} 
    is therefore completed.
    }
\qed

\begin{remark}
Theorem \ref{2023SIS-thm:persistence_scheme} illustrates that, when the scheme parameters $\alpha, \theta$ obey 
the condition \eqref{eq:thm-persistence-alpha-condition},
the numerical approximation produced by the LCM scheme with any step size $h>0$ is able to reproduce the persistence property of the original model without any addition restriction on the model parameters. It is worthwhile to point out that, the condition \eqref{eq:thm-persistence-alpha-condition} can be easily fulfilled,
by taking $\alpha $ to be sufficiently small.
\end{remark}

\section{Numerical experiments}
\label{2023SIS-section:numerical_experiments}

In this section, we provide numerical experiments to illustrate the previous theoretical findings.
Consider the stochastic SIS epidemic model
\begin{equation*}
    \d I_t
    =
    I_t (\beta N - \mu - \gamma - \beta I_t)\, \d t
    +
    \sigma I_t (N - I_t) \, \d W_t,
    \quad
    0 \leq t \leq T.
\end{equation*}
The population sizes are measured in units of 1 billion and the unit of time is assumed to be one day throughout this section.

The approximation errors will be calculated in terms of 
$$
\bigg( \E \Big[ 
    \sup\limits_{k=0,...,M} \big\vert I_{t_k} - \mathcal{I}_k \big\vert^2
\Big] 
\bigg)^{\frac{1}{2}}.
$$
The newly proposed LCM scheme \eqref{2023SIS-eq:SIS_log_corrected_Mil_scheme},
and the {\color{black}Lamperti truncated Euler-Maruyama (LTEM)} scheme proposed in {\color{black}\cite{yang2023strong}}
{\color{black}with the scheme parameter $K = (1 + e^{Y_0}) + 50$ 
}
will be 
tested for comparison.
{\color{red}Note that the specific value of the scheme parameter $K$ was not explicitly stated in \cite{yang2023strong}. In the numerical experiments, we choose $K = 1 + e^{Y_0} + 50$ such that $K > 1 + e^{Y_0}$, which was required by assumptions in \cite{yang2023strong}}.
Moreover, we use numerical approximations {\color{black}produced by LTEM, which has been proved to convergent with order one,} with the step size $h_{exact} = 2^{-14}$ to identify the "exact" solutions
and those with step sizes $h = 2^{-i}, i=4,5,6,7,8$ for numerical schemes. The expectations appearing in the errors are approximated by calculating averages over $10000$ paths.
The following two sets of parameters, with fixed $N = 10$, are taken for tests.

\begin{example}\label{2023SIS-eg:convergence_1}
    $T = 1, \beta = 0.5, \sigma = 0.2, \mu + \gamma = 4, \alpha = 0.1, \theta = 2$, {\color{red}$I_0 = 1$}.
\end{example}

\begin{example}\label{2023SIS-eg:convergence_2}
    $T = 1, \beta = 0.7, \sigma = 0.1, \mu + \gamma = 2, \alpha = 1, \theta = 3$ {\color{red}$I_0 = 9$}.
\end{example}

\begin{figure}[htp]
    \centering
    \begin{minipage}{0.48\linewidth}
        \centering
        \includegraphics[width = 1\linewidth]{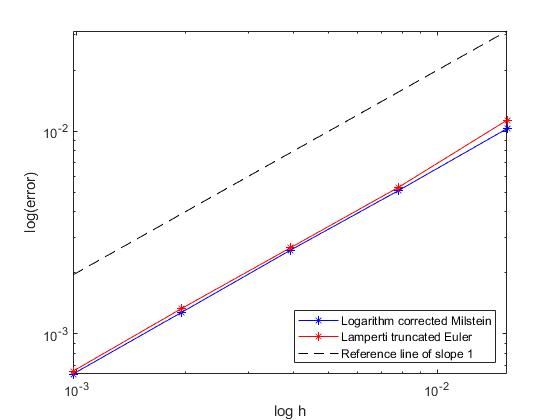}
        \caption{\centering Example \ref{2023SIS-eg:convergence_1} }
        \label{2023SIS-fig.1}
    \end{minipage}
    \begin{minipage}{0.48\linewidth}
        \centering
        \includegraphics[width = 1\linewidth]{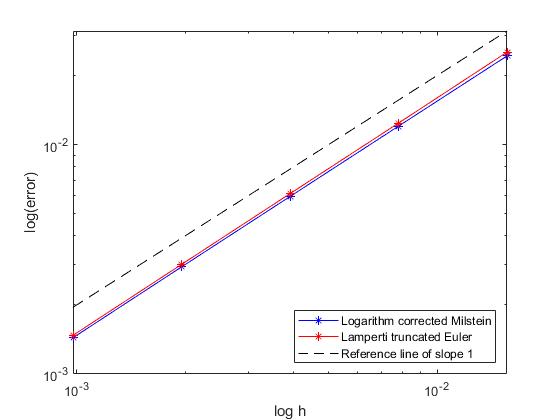}
        \caption{\centering Example \ref{2023SIS-eg:convergence_2} }
        \label{2023SIS-fig.2}
    \end{minipage}
\end{figure}

\begin{table}[htp]
    \footnotesize    
    \caption{\centering{\bf  Sample average errors}}
    \label{2023SIS-tab:error}
    \begin{tabular}
        {p{1cm}<{\centering} p{1.5cm}<{\centering}  p{1.5cm}<{\centering} p{1.5cm}<{\centering} p{1.5cm}<{\centering} p{1.5cm}<{\centering} p{1.5cm}<{\centering}}
        \toprule   & $h$ & $2^{-10}$ & $2^{-9}$ & $2^{-8}$ & $2^{-7}$ & $2^{-6}$\\
        \midrule 
        \multirow{2}*{ Ex \ref{2023SIS-eg:convergence_1}} 
        & {\color{black}LCM} & {\color{black}0.0006} & {\color{red}0.0013} & {\color{red}0.0026} & {\color{red}0.0051} & {\color{red}0.0103}\\
        & {\color{black}LTEM} & {\color{black}0.0007} & {\color{red}0.0013} & {\color{red}0.0027} & {\color{red}0.0053} & {\color{red}0.0113}\\
        \midrule  
        \multirow{2}*{Ex \ref{2023SIS-eg:convergence_2}} 
        & {\color{black}LCM} & {\color{red}0.0014} & {\color{red}0.0029} & {\color{red}0.0059} & {\color{red}0.0120} & {\color{red}0.0243}\\
        & {\color{black}LTEM} & {\color{black}0.0015} & {\color{red}0.0030} & {\color{red}0.0061} & {\color{red}0.0124} & {\color{red}0.0253}\\
    \bottomrule
\end{tabular}
\end{table}

\begin{table}[htp]
    \footnotesize
    \caption{ \centering{ \bf Least-squares fit for the convergence rate $q$}}
    \label{2023SIS-tab:convergence_rate}
    \centering
    \begin{tabular}
    {m{2cm}<{\centering} m{4cm}<{\centering} m{4cm}<{\centering} }
    \toprule  & {\color{black}LCM} scheme & {\color{black}LTEM} scheme \\
    \midrule 
    Ex \ref{2023SIS-eg:convergence_1} & {\color{red}$q = 1.0047$, resid = $0.0120$}& {\color{red}$q = 1.0223$, resid = 0.0417}\\
    Ex \ref{2023SIS-eg:convergence_2} & {\color{red}$q = 1.0198$, resid = 0.0032}& {\color{red}$q = 1.0245$, resid = 0.0016}\\
    \bottomrule
\end{tabular}
\end{table}

Figs \ref{2023SIS-fig.1} and \ref{2023SIS-fig.2} show the average sample errors against various step sizes on a log-log scale. Detailed data is presented in Table \ref{2023SIS-tab:error}.
The approximation errors of the LCM and {\color{black}LTEM} 
scheme are plotted in the blue and red 
solid lines. The black dashed lines are reference lines of slope 1.
From the figures we can easily identify convergence rates of order 1 for both the LCM and {\color{black}LTEM} scheme.
Moreover, by assuming that $error \approx C h^q$ such that $\log (error) \approx \log C + q \log h$, the convergence rate $q$ and the least square residual is obtained with a least-squares fitting as presented in Table \ref{2023SIS-tab:convergence_rate}.
These results confirm the expected convergence rate.
    
Concerning the dynamics behavior of the approximations, we test three different schemes for comparison: 
the newly proposed LCM scheme,
the {\color{black}LTEM} scheme and 
{\color{black}the usual Milstein scheme \cite{mil1975approximate}}
directly applied to the original SDE without any transformation.
The following two sets of model parameters, carefully chosen to meet all conditions in Theorems \ref{2023SIS-thm:extinction_scheme} or \ref{2023SIS-thm:persistence_scheme}, with fixed $\alpha = 0.1$ and $\theta = 2$, are taken as examples.

\begin{example}\label{2023SIS-eg:extinction_2}
    {\color{black}$\beta = 0.42, \sigma = 0.9, \mu + \gamma = 10, I_0=90, N=100$}.
\end{example}

\begin{example}\label{2023SIS-eg:persistence_2}
    ${\color{black}\beta = 0.6, \sigma = 0.01, \mu + \gamma = 40, I_0 = 10, N = 100}$.
\end{example}

We mention that parameters in Example 
\ref{2023SIS-eg:extinction_2} satisfy the conditions of extinction in Theorem \ref{2023SIS-thm:extinction_SIS_original} and \ref{2023SIS-thm:extinction_scheme}, 
while those in Example 
\ref{2023SIS-eg:persistence_2} meet the requirements of persistence in Theorem \ref{2023SIS-thm:persistence_original} and \ref{2023SIS-thm:persistence_scheme}.
Moreover, we can conclude by simple calculations that
\vspace{-5pt}
\begin{equation*}
    \liminf\limits_{k h \rightarrow \infty} \mathcal{I}_k 
    \leq 
    {\color{black}32.9588}
    \leq
    \limsup\limits_{k h \rightarrow \infty} \mathcal{I}_k
    \quad
    a.s.
\end{equation*}
\vspace{-6pt}
when the Example \ref{2023SIS-eg:persistence_2} is under consideration.

\FloatBarrier

\begin{figure}[htp]
    \centering
        \includegraphics[width = 0.6\linewidth]{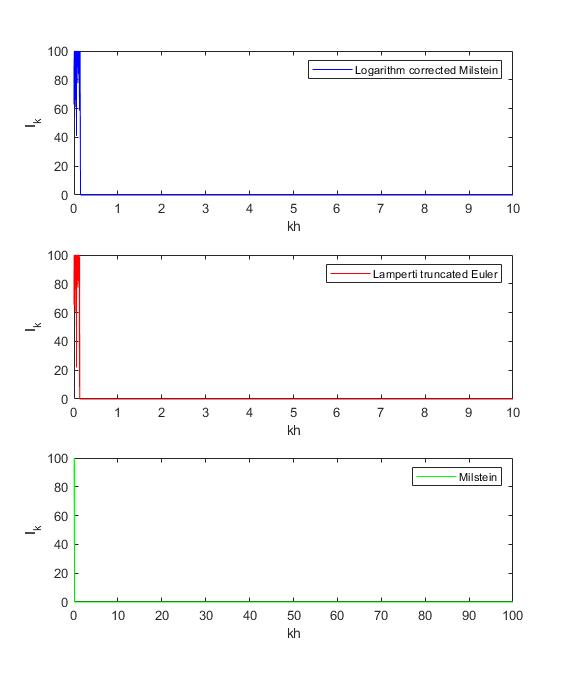}
        \caption{\centering Example \ref{2023SIS-eg:extinction_2}: {\color{black}$h_{\text{exact}} = 2^{-14}$} }
        \label{2023SIS-fig:extinction_2_exact}
\end{figure}

\FloatBarrier

\begin{figure}[htp]
    \centering
    \begin{minipage}{0.48\linewidth}
        \includegraphics[width = 1\linewidth]{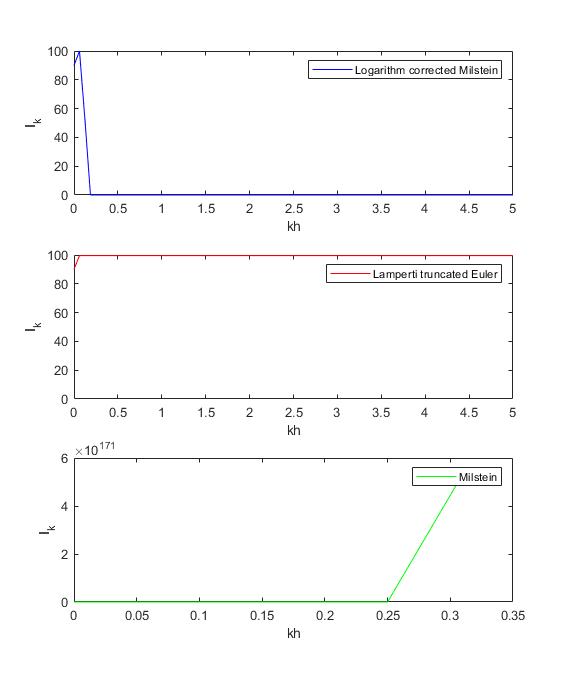}
    \caption{\centering{Example \ref{2023SIS-eg:extinction_2}: {\color{black}$h = 2^{-4}$} }}
    \label{2023SIS-fig:extinction_2_4}
    \end{minipage} 
    \begin{minipage}{0.48\linewidth}
        \includegraphics[width = 1\linewidth]{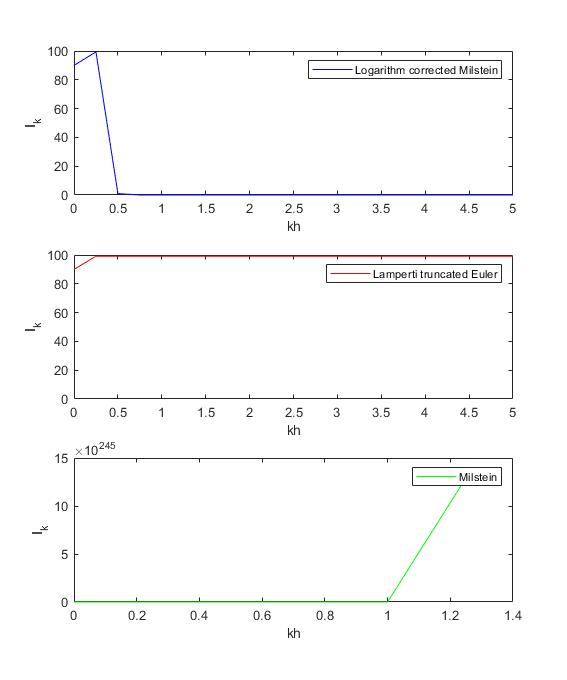}
        \caption{\centering{Example \ref{2023SIS-eg:extinction_2}: {\color{black}$h = 2^{-2}$} }}
    \label{2023SIS-fig:extinction_2_2}
    \end{minipage}
\end{figure}

\begin{figure}[htp]
    \centering
        \includegraphics[width = 0.6\linewidth]{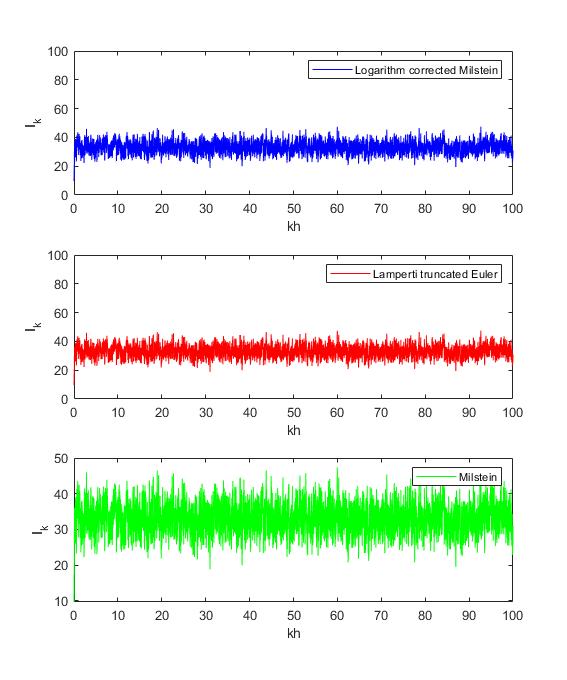}
        \caption{ \centering{ Example \ref{2023SIS-eg:persistence_2}: $h_{\text{exact}}=2^{-14}$  }}
        \label{2023SIS-fig:persistence_2_exact}
\end{figure}

\begin{figure}[htp]
    \centering
    \begin{minipage}{0.48\linewidth}
        \includegraphics[width = 1\linewidth]{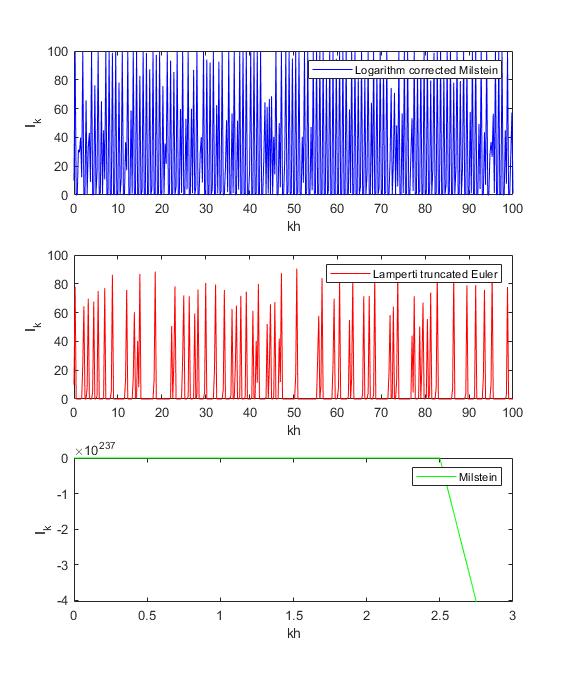}
        \caption{\centering{Example \ref{2023SIS-eg:persistence_2}: $h=2^{-2}$ }}
    \label{2023SIS-fig:persistence_2_2}
    \end{minipage}
    \begin{minipage}{0.48\linewidth}
        \includegraphics[width = 1\linewidth]{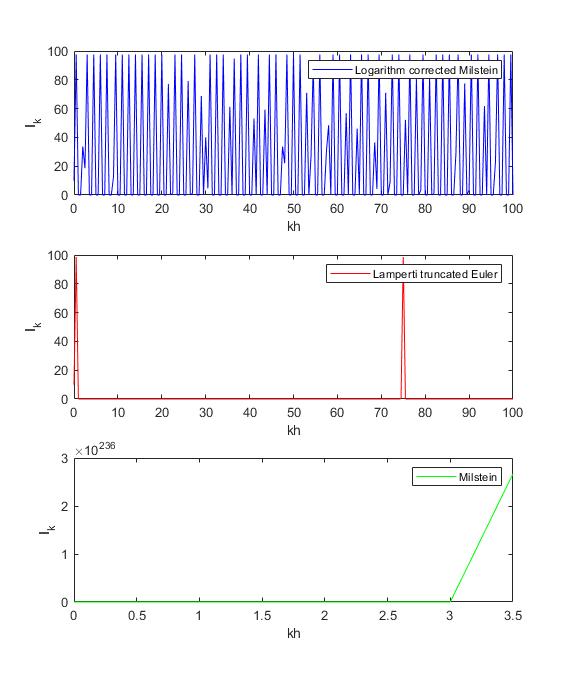}
        \caption{\centering{Example \ref{2023SIS-eg:persistence_2}: $h=2^{-1}$ }}
    \label{2023SIS-fig:persistence_2_1}
    \end{minipage}
\end{figure}

\FloatBarrier

\begin{table}[htp]
    \footnotesize  
    \color{red}
    \caption{\centering{\bf \color{red} The average percentage of truncation over $2 \times 10^4$ paths}}
    \label{2023SIS-tab:P_Y>logN}
    \begin{tabular}
        {p{1cm}<{\centering} p{1.5cm}<{\centering}  p{1.5cm}<{\centering} p{1.5cm}<{\centering} p{1.5cm}<{\centering} p{1.5cm}<{\centering} p{1.5cm}<{\centering}}
        \toprule  & $h$ & $2^{-5}$ & $2^{-4}$ & $2^{-3}$ & $2^{-2}$ & $2^{-1}$\\
        \midrule 
        \multirow{3}*{Set (1)} & $I_0 = 10$ & $0$ & $0$ & $0$ & $0$ & $0$ \\
        & $I_0 = 50$ & $0.0072\%$ & $0.0081\%$ & $0.0129\%$ & $0.0508\%$ & $0.1199\%$ \\
        & $I_0 = 90$ & $0.0488\%$ & $0.0578\%$ & $0.0977\%$ & $0.3927\%$ & $0.9408\%$ \\
        \midrule
        \multirow{3}*{Set (2)} & $I_0 = 10$ & $0$ & $5.5191\%$ & $49.9986\%$ & $50\%$ & $50\%$ \\
        & $I_0 = 50$ & $0$ & $5.5430\%$ & $49.9985\%$ & $50\%$ & $50\%$ \\
        & $I_0 = 90$ & $0$ & $5.4993\%$ & $49.9986\%$ & $50\%$ & $50\%$ \\
        \midrule
        \multirow{3}*{Set (3)} & $I_0 = 10$ & $0$ & $0$ & $0.5397\%$ & $19.5684\%$ & $25.0898\%$ \\
        & $I_0 = 50$ & $0$ & $0$ & $0.5382\%$ & $19.5228\%$ & $24.9660\%$ \\
        & $I_0 = 90$ & $0$ & $0$ & $0.5379\%$ & $19.4598\%$ & $24.8422\%$ \\
        \midrule
        \multirow{3}* {Set (4)} & $I_0 = 10$ & $0.0007\%$ & $0.0015\%$ & $0.0009\%$ & $0.0006\%$ & $<0.0001\%$ \\
        & $I_0 = 50$ & $0.0006\%$ & $0.0011\%$ & $0.0011\%$ & $0.0008\%$ & $0.0004\%$ \\
        & $I_0 = 90$ & $0.0006\%$ & $0.0013\%$ & $0.0015\%$ & $0.0003\%$ & $<0.0001\%$ \\
    \bottomrule
\end{tabular}
\end{table}

\FloatBarrier

{\color{red}We conduct numerical experiments on one-path simulations for the three different schemes with various step sizes $h=2^{-i},i=0,1,...,5,14$.
In Figures \ref{2023SIS-fig:extinction_2_exact}-\ref{2023SIS-fig:persistence_2_1}, we only show simulations with some chosen stepsizes.}
{\color{red}More accurately, Figure \ref{2023SIS-fig:extinction_2_exact} presents numerical approximations to the SDE produced by three different schemes with a very fine step size $h_{\text{exact}} = 2^{-14}$, which reveal the exact behavior of the exact solution.
This allows us to tell which scheme is most accurate and reliable when used to approximate the {\color{green}dynamics behavior} of the model with larger and usual step sizes later.}
{\color{red}Our numerical results show, LCM exactly preserves the extinction property with six different step sizes, as opposed to explosion of LTEM  when $h \geq 2^{-4}$, which can be also observed from Figures \ref{2023SIS-fig:extinction_2_4}-\ref{2023SIS-fig:extinction_2_2} for $h = 2^{-4}$ and $h = 2^{-2}$.
}
For Example \ref{2023SIS-eg:persistence_2}, Figure \ref{2023SIS-fig:persistence_2_exact} provides a glance at the {\color{green}dynamics behavior} of the "exact" solutions.
{\color{red}From Figures \ref{2023SIS-fig:persistence_2_2}-\ref{2023SIS-fig:persistence_2_1} one can observe that}
the newly proposed LCM scheme reveals a persistent property for all step sizes.
Instead, the LTEM scheme 
for Example \ref{2023SIS-eg:persistence_2}
can only preserve the persistence property for sufficiently small step sizes {\color{red}$h \leq 2^{-2}$} 
but fails with
step sizes {\color{red}$ h \geq 2^{-1}$}.
The Milstein scheme even explodes for all chosen step sizes {\color{red}except for $h = 2^{-14}$}.
These results confirm the excellent dynamics-preserving properties of the LCM scheme we construct.

{\color{red}
Next, we are concerned with 
the truncation frequency of the LCM method, which would lead to a bias and is expected to be heavily dependent on the parameters and step sizes. 
To this end, we fix $\alpha = 0.1$, $\theta = 2$ and take four sets of parameters as follows:
\begin{enumerate}[(1)]
    \item $(R_0^S = -400.8)$ $\beta = 0.42$,  $\mu + \gamma = 10$, $\sigma = 0.9$,
    \item $(R_0^S = 4.15)$ $\beta = 0.42$,  $\mu + \gamma = 10$, $\sigma = 0.01$,
    \item $(R_0^S = 1.4875)$ $\beta = 0.6$,  $\mu + \gamma = 40$, $\sigma = 0.01$,
    \item $(R_0^S = 0.25)$ $\beta = 0.6$,  $\mu + \gamma = 40$, $\sigma = 0.1$,
\end{enumerate}
with different initial values $I_0 = 10, 50, 90$ and five different step sizes $h = 2^{-i}, i = 1,2,...,5$.
We run $2 \times 10^4$ trajectories and list the average percentage of truncation over $2 \times 10^4$ paths
in Table \ref{2023SIS-tab:P_Y>logN}.
Numerical results indicate that the truncation frequency is extremely low for small step sizes ($h = 2^{-4}, 2^{-5}$). 
Furthermore, the truncation frequency for the extinction cases (1) and (4) is significantly lower  than that for the persistence cases (2) and (3). 
For both cases, the truncation frequency tends to zero, as the step size $h$ shrinks.

}

\section{Conclusion}\label{2023SIS-section:conclusion}

In this paper we propose a first-order strongly convergent scheme for the stochastic SIS epidemic model which preserves the domain as well as the dynamics behavior of the model unconditionally.
{\color{black}
The easily implementable scheme relies on 
a logarithm transformation combined with a corrected explicit Milstein-type method. 
The strong convergence rate of the scheme is carefully analyzed and proved to be order $1$. 
Moreover, the proposed scheme is able to reproduce the dynamics behavior, namely, the extinction and persistence properties of the 
original model without additional requirements on the model parameters and the step size $h>0$ .
Numerical experiments are provided to verify the convergence analysis and comparisons of dynamics behavior between different schemes are presented to show the advantage of the proposed scheme.
Before closing the conclusion section, we mention that higher-order schemes which are also able to reproduce the dynamic properties of the considered model, as studied by \cite{LIU2023107258}, are on the top of the list of our future works.



}

\backmatter

\bmhead{Acknowledgments}
The authors thank the associated editor and anonymous reviewers for the helpful comments and suggestions.

\section*{Declarations}

\subsection*{Funding}
This work is supported by Natural Science Foundation of China (12471394, 12071488, 11971488) and the Fundamental Research Funds for the Central Universities of Central South University (Grant No. 2023zzts0348).

\subsection*{Conflict of interest/Competing interests}
No potential conflict of interest was reported by the authors.

\subsection*{Ethical Approval} 
Not applicable

\subsection*{Consent to participate}
Not applicable

\subsection*{Consent for publication}
All authors have approved the manuscript for its submission and publication.

\subsection*{Availability of data and materials}
The authors confirm that the data and  materials supporting the findings of this study are available within the article.

\subsection*{Code availability }
Code will be made available on request.

\subsection*{Authors' contributions} 
Ruishu Liu:  Formal analysis, Visualization, Writing - Original Draft, Software. Xiaojie Wang: Conceptualization, Methodology, Writing - Review \& Editing, Funding acquisition, Supervision. Lei Dai: Methodology, Validation, Investigation, Formal analysis.

\bibliography{SIS_Reference}

\end{document}